\documentclass[12pt]{amsart}

\usepackage{amssymb}

\def\H{\mathcal{H}}
\def\i{\mathrm{i}}
\def\im{\operatorname{Im}}
\def\re{\operatorname{Re}}
\def\B{\mathcal{B}}
\def\A{\mathcal{A}}
\def\L{\operatorname{L}}
\def\sh{\operatorname{sh}}
\def\Sh{\operatorname{Sh}}
\def\M{\mathrm{M}}
\def\Mm{\mathcal{M}}
\def\Linf{\L^\infty(\S,\mathbb{C})}
\def\Ltwo{\L^2(\S,\mathbb{C})}
\def\sgn{\operatorname{sgn}}
\def\m{\operatorname{m}}
\def\id{\operatorname{id}}
\def\tr{\operatorname{tr}}
\def\S{\mathrm{S}^1}

\def\CC{\mathbb{C}}

\def\G{\mathcal{G}}

\newtheorem{lemma}{Lemma}[section]
\newtheorem{theorem}[lemma]{Theorem}
\newtheorem{corollary}[lemma]{Corollary}
\newtheorem{proposition}[lemma]{Proposition}
\newtheorem{remark}[lemma]{Remark}

\newtheorem{question}[lemma]{Question}

\title{Invariant embeddings and ergodic obstructions}

\author{Mitja Mastnak}
\address{Department of Mathematics and Computing Science, Saint Mary's University, 923 Robie St, Halifax, Nova Scotia, Canada B3N 1Z9}
\email{mitja.mastnak@smu.ca}

\author{Heydar Radjavi}
\address{Department of Pure Mathematics, University of Waterloo, Waterloo, Ontario,
Canada N2L 3G1}
\email{hradjavi@uwaterloo.ca}



\date{\today}

\begin{document}

\begin{abstract} We consider the following question: Let $\A$ be an abelian self-adjoint algebra of bounded operators on a Hilbert space $\H$. Assume that $\A$ is invariant under conjugation by a unitary operator $U$, i.e., $U^* AU$ is in $\A$ for every member $A$ of $\A$. Is there a maximal abelian self-adjoint algebra containing $\A$, which is still invariant under conjugation by $U$?  The answer, which is easily seen to be yes in finite dimensions, is not trivial in general. We prove affirmative answers in special cases including the one where $\A$ is generated by a compact operator.  We also construct a counterexample in the general case, whose existence is perhaps surprising.
\end{abstract}



\maketitle
\section{Introduction}
Our general problem of interest can be stated in the following way.  Consider an arbitrary abelian algebra $\mathcal{A}$ of linear operators on a vector space $\mathcal{V}$.  Assume that $\mathcal{A}$ is mapped to itself by a given bijective linear isometry of the underlying space, that is, there exists a linear isometry $U\colon \mathcal{V}\to\mathcal{V}$ satisfying $U^{-1} A U\in\mathcal{A}$ for every $A\in\mathcal{A}$.  The question is: can $\A$ be extended to a maximal abelian algebra $\widehat{\A}\supseteq\A$ which is still mapped into itself by the same isometry $U$. The answer is easily seen to be yes if $U$ belongs to $\A$, but not in general as will be shown.

Perhaps it should be emphasized here that the question is being asked about a given, fixed isometry $U$, and not just about its effect on the given algebra $\A$.  In other words, the question is much easier to answer if we allow an isometry $U_0$ and a maximal abelian algebra $\widehat{\A}$ such that $\A\subseteq\widehat{\A}$,
$$ U_0^* A U_0=U^* A U\mbox{ for all } A\in\A,\mbox{ and } U_0^*\widehat{\A} U_0\subseteq \widehat{\A}.$$
To see this, just observe the following: the commutative $C^*$-algebra $\A$ is singly generated.  Let $A$ be a generator.  By the Spectral Theorem \cite[Theorem 4.6]{C}, $A$ can be expressed as a multiplication operator $M_\varphi$ on a space $\L^2(X,\mu)$ with a measure $\mu$, i.e., there is an $\L^\infty$ function $\varphi$ such that $(Af)(x)=(M_\varphi f)(x)=\varphi(x)f(x)$ almost everywhere.  Then, since $U^*\A U\subseteq \A$, we have that $U^* A U=M_\psi$ with $\psi\in\L^\infty(X,\mu)$.  Now suppose that there exists a ``permutation" of the underlying space $X$, i.e., a measure-preserving transformation $\tau$ of $X$ onto itself such that $\psi=\varphi\circ\tau$ (such a $\varphi$ clearly cannot not always exits, but it does for our ergodic counterexample discussed in Section 3).   Now consider the maximal extension 
$$\widehat{\A}=\left\{M_\xi: \xi\in\L^\infty(X,\mu)\right\}$$
and define $U_0$ on $\L^2(X,\mu)$ by $U_0 f = f\circ\tau^{-1}$.  Then, for every $M_\xi$ in $\widehat{\A}$ we have 
$$ U_0^* M_\xi U_0=M_{\xi\circ\tau}$$
and $U_0^* T U_0=U^* T U$ for all $T$ in the original $\A$.  Thus the ``effect" of conjugation by $U$ has been extended from $\A$ to $\widehat{A}$.

Returning to our current problem, let us note that with the original $\A$ and $U$ with the newly defined $U_0$, the operator 
$$
V=U U_0^*
$$
commutes with every member of $\A$.  Thus $U=V U_0$, where $V$ is in the commutant of $\A$ and $U_0$ is a composition by a measure-preserving bijection.  It is interesting to observe that if either of the two unitaries $U_0$ and $V$ is scalar, the problem is trivialized.

More general questions can be asked. For example, instead of a single isometry, one can consider a whole group $\G$ of isometries and ask the same question under the hypothesis $U^{-1} A U\in \A$ for all $U\in \G$ and all $A\in \A$ (see, e.g., \cite{MR}).

In this paper, the underlying space will be a complex separable Hilbert space $\H$, and $\A$ will be an abelian self-adjoint subalgebra of $\B(\H)$, the algebra of all bounded linear operators on $\H$.
We show that under certain conditions, the assumption $U^*\A U$ implies the existence of a maximal abelian self-adjoint algebra (masa) $\widehat{A}\supseteq\A$ with $U^*\widehat{\A} U\subseteq \A$ as desired.  In particular, if $\A$ is generated by its compact operators, this is true (as it is also 
in the special case of finite dimensional $\H$ \cite{MR}).

We also present a, perhaps unexpected, counterexample in general.  In this counterexample
ergodic operators play an important role.

\section{Invariant Embedding Theorem}
Throughout the section $\H$ will denote an infinite dimensional separable Hilbert space.  If $U\in \B(\H)$ is unitary and $\A\subseteq\B(\H)$, then we say that $\A$ is $U$-invariant, or invariant under conjugation by $U$ if
$U^*\A U\subseteq \A$.  We first observe, that this assumption of inclusion can, with no loss of generality, be replaced with equality.

\begin{proposition} Let $\A$ be a $U$-invariant abelian self-adjoint algebra.  Then $\A$ can be embedded into an abelian self-adjont algebra $\widehat{A}$ such that $U^*\widehat{\A} U=\widehat{A}$. If $\A$ is generated by its compact operators, we can also require this for $\widehat{A}$.  
\end{proposition}
\begin{proof}
Take $\widehat{A}=\bigcup_{n=1}^\infty U^n\A (U^*)^n$.
\end{proof}

\begin{remark} If $U$ is such that $U^*$ is in the weak operator closure of the algebra generated by $U$ (for, example, if $U$ is discrete in the sense that its eigenvectors generate $\H$), then clearly $U^*\A U\subseteq \A$ implies that $U^*\A U= \A$.  However, even if $U$ is discrete and multiplicity free, it does not follows that $\A$ is discrete.  For example, let $\H=\ell^2(\mathbb{Z})$, let $\A$ be generated by the bilateral shift (i.e., operator $A$ given by $A e_i=e_{i+1}$) and $U$ be the diagonal operator given by $Ue_i=t^i e_i$, where $t\in\CC$ is some fixed transcendental number.
\end{remark}

Throughout the section $(X,\mu)$ will denote a $\sigma$-finite complete measure space. Let $\H=\L^2(X,\mu)$.  Let $\varphi\colon X\to X$ be an almost bijection such that the measures
$\mu$ and $\nu=\mu\circ\varphi^{-1}$ are equivalent (i.e., mutually absolutely continuous) and $\varphi$ is an almost isomorphism between measure spaces $(X,\nu)$ and $(X,\mu)$.  To spell it out, $\varphi$ preserves measurable sets and null-sets in both directions, i.e., for every set $A\subseteq X$ we have that $A$ is measurable if and only if $\varphi^{-1}(A)$ is, and $A$ is a null-set if and only $\varphi^{-1}(A)$ is.   Let $h=\frac{d\mu}{d\nu}\colon X\to[0,\infty)$ be the Radon-Nikodym derivative of $\mu$ with respect to $\nu$, i.e., $h$ is a measurable function such that for every measurable set $A$ we have that $\mu(A)=\int_A h d\nu$.  Then the weighted composition operator $W=W_{h,\varphi}\in\B(\H)$ given by 
$$W f = (\sqrt{h} f)\circ\varphi $$
is unitary.  Indeed, if $f\in \H$, then
\begin{eqnarray*}
|| W f ||^2 &=& \int_X |(\sqrt{h} f)\circ\varphi|^2 d\mu \\
&=& \int_X (h |f|^2)\circ\varphi  d\mu \\
&=& \int_X h|f|^2 d \nu \\
&=& \int_X |f|^2 d\mu = ||f||^2
\end{eqnarray*}
We note that if $f,g\in\L^\infty(X,\mu)$ are such that $g=f\circ\varphi$ a.e., then the multiplication operators $\M_f$ and $\M_g$ are unitarily equivalent.  Indeed, if $k\in\L^2(X,\mu)$, then
$$
(W \M_f) k = (\sqrt{h} f k)\circ \varphi = g((\sqrt{h}k)\circ\varphi)  
= \M_g (W k).
$$
In general the converse of the above fails.  For example, if $X$ is the disjoint union of an 
infinite discrete measure space $X_0$ (say $\mathbb{N}$ with counting measure) and a nontrivial continuous measure space $X_1$ (say $[0,1]$ with Lebesgue measure), then for $f=\chi_{X_0}$ and $g=\chi_{X_1}$, the multiplication operators $\M_f$ and $\M_g$ are clearly unitarily equivalent (they are projections of infinite rank and co-rank), but clearly there is no $\varphi$ as above such that $g=f\circ\varphi$ (in fact, there is no almost bijection such that $g=f\circ\varphi$ a.e.).
However, if $(X,\mu)$ is discrete, then the converse holds as well.
\begin{proposition}\label{prop-bijection}
Let $(X,\mu)$ be a countable discrete measure space, let $\H=\L^2(X,\mu)$, and let $f,g\in\L^\infty(X,\mu)$.  If the multiplication operators $\M_f,\M_g$ are unitarily equivalent, then there is
a bijection $\varphi\colon X\to X$ such that $g=f\circ\varphi$.  
\end{proposition}
\begin{proof} This immediately follows from the fact that for each element $\lambda$ in the range of $f$ we have that the cardinality of $f^{-1}(\lambda)$ is the dimension of the eigenspace of the eigenvalue $\lambda$ of $\M_f$ and hence must be the same as the cardinality of $g^{-1}(\lambda)$, the dimension of the corresponding eigenspace for $\M_g$.
\end{proof}

\begin{question} Does the analogous result also hold for ``nice" continuous measure spaces (e.g., standard probability spaces)?
\end{question}

\begin{theorem}[Invariant Masa Embedding Theorem.]\label{thm-invariant}
Let $\H=L^2(X)$ where $X$ is a discrete measure space, let $U\in\B(\H)$ be a unitary operator and let $\A$ be an abelian self-adjoint algbera that is contained in the multiplier algebra $\M(X)\simeq\L^\infty(X)$ of multiplication operators.  If $U^*\A U=\A$, then $\A$ can be embedded into a $U$-invariant masa.
\end{theorem}
\begin{proof}
Note that we have a disjoint union decomposition 
$$
X=\bigcup_{j\in J} X_j,
$$
such that
$$
\A=\left\{M_f : f\mbox{ is constant on each } X_j\right\}. 
$$
We write $U=VW$, where $V$ is in the commutant of $\A$ (or, equivalently, has block decomposition $V=\bigoplus_{j\in J} V_j$, where
$V_j\in\mathcal{B}{L^2(X_j)}$, and $W$ is the weighted composition operator with a bijection $\varphi\colon X\to X$ (i.e., $W f = (\sqrt{h}f)\circ\varphi$, where $h$ is the Radon-Nikodym derivative of $\mu$ with respect to $\nu=\mu\circ\varphi^{-1}$)  Observe that, since $X$ is discrete, we have that for all $j,k\in J$, 
$$\left(\varphi(X_j)\cap X_j\not=\emptyset\right)\iff \varphi(X_j)=X_k.$$
With this in mind, $\varphi$ induces a bijection $\pi\colon J\to J$ by $$\pi(j)=k \iff \varphi(X_j)=X_k.$$
This bijection splits $J$ into a disjoint union of ``cycles":
$$
J=\bigcup_{k\in K} J_k,
$$
where for each $k$, $\pi(J_k)=J_j$ and either $|J_k|=1$, or $\pi$ has no fixed points on $J_k$.  For each $k$, let $n_k=|J_k|$ (possibly infinity) and pick an element $a^{(0)}_k\in J_k$ and label the other elements in $J_k$ as follows:
\begin{itemize}
\item If $|J_k|=1$, then there are no other elements.
\item If $1<|J_k|<\infty$, then let $a^{(r)}_k=\pi^r(a^{(0)}_k)$ for $r\in\mathbb{Z}/n_k \mathbb{Z}$.
\item If $|J_k|=\infty$, then let $a^{(r)}_k=\pi^r(a^{(0)}_k)$ for $r\in\mathbb{Z}$.
\end{itemize}
Here, for $r>0$, $\pi^r=\pi\circ\pi\circ\ldots\circ\pi$ is the $r$-fold composition and for $r<0$ we have $\pi^r=(\pi^{-1})^{|r|}$ is the $|r|$ fold composition of the inverse $\pi^{-1}$ of $\pi$.  Now we define maximal abelian self-adjoint algebras $\A_j\subseteq \mathcal{B}(L^2(X_j))$ for $j\in J$ as follows.  

Let $k\in K$.  Then
\begin{itemize}
\item If $|J_k|=n_k<\infty$, then observe that
$\mathcal{B}(L^2(X_{a^{(0)}_k})$ is invariant under conjugation by $U^{n_k}$. Let $\mathcal{B}_k$ be a $U^{n_k}$-invariant masa in 
$\mathcal{B}(L^2(X_{a^{(0)}_k})$. For $r=0,\ldots, n_k$ we then define
$\mathcal{A}_{a^{(r)}_k}=(U^*)^r \mathcal{B}_k U^r\subseteq \mathcal{B}(L^2(X_{a^{(r)}_k})$.
\item If $|J_k|=\infty$, then let $\mathcal{B}_k$ be any masa in 
$\mathcal{B}(L^2(X_{a^{(0)}_k})$.  For For $r\in\mathbb{Z}$ we then define
$\mathcal{A}_{a^{(r)}_k}=(U^*)^r \mathcal{B}_k U^r\subseteq \mathcal{B}(L^2(X_{a^{(r)}_k})$.
\end{itemize}
This defines $\A_j$, since every $j\in J$ can be written uniquely as one of $a^{(r)}_k$ above.

We now note that $\B=\bigoplus_{j\in J} \A_j$ is a $U$-invariant masa.

\end{proof}

\begin{remark} In the proof of the theorem above, it is crucial to have equality $U^*\A U=\A$.  It is fairly easy to see that, in general, it is not possible to embed a $U$-invariant discrete $\A$ into a discrete $\widehat{A}$ satisfying $U^*\widehat{\A} U=\widehat{A}$.  
Also note that the masa $\widehat{\A}$ we construct in the above proof, is in general not discrete.
At this point we do not know whether the weaker condition $U^*\A U\subseteq \A$ for a discrete $\A$ also yields the same conclusion.
\end{remark}

\begin{corollary} Let $U\in \B(\H)$ be a unitary operator and let $\A\subseteq\B(\H)$ be an abelian self-adjoint algebra that is invariant under conjugation by $U$.  If $\A$ is the weak operator-closure of the subalgebra of its compact operators, then there exists a masa $\B\subseteq\B(\H)$ that contains $\A$ and is also invariant under conjugation by $U$.
\end{corollary}
\qed

\section{Ergodic counterexample}

In this section we give an example of a unitary $U\in\B(\H)$ and a $U$-invariant abelian self-adjoint algebra $\A$ (in fact $U^*\A U=\A$) that does not embed into any $U$-invariant masa. The main ingredient of the proof is the ergodicity of irrational shift operators.  For the general theory of ergodic maps we refer to the Section 1.5 of \cite{W}.

Let $\S=\mathbb{R}/\mathbb{Z}$.  We often identify it with $[0,1)$ and also think of it as a circle of circumference $1$.  Let $\mathcal{H}_0=\Ltwo$ and let $\mathcal{H}=\mathcal{H}_0\oplus\mathcal{H}_0$.  In what follows we will identify $\B(\H)$ with $\mathcal{M}_2(\B(\H_0))$.   For any $a\in [0,1)$ we denote by $\sh_a\colon\S \to\S$ the shift operator given by $\sh_a(x)=a+x\,(\mbox{mod }\mathbb{Z})$ (if we think of $\S$ as a circle, then this is a rotation by angle $2\pi a$).  We will use the fact that $\sh_a$ is ergodic, i.e., that there are no nontrivial essentially invariant subsets (since $\sh_a$ is bijective, this means that whenever a measurable set $X\subseteq\S$ is such that the measure of the symmetric difference $X\triangle \sh_a(X)$ is zero, then one of the sets $X$, $\S\setminus X$ must have measure $0$).  We will also need the fact that powers of $\sh_a$ and those of
the relative shift $\sh'_1\colon [0,a)\to [0,a)$, $t\mapsto t+1 \mod a$, 
are ergodic as well (see the last two paragraphs of the proof of Theorem \ref{thm-nonexistence} for more detail).

We denote by $\Sh_a\in\B(\H_0)$ the composition by the above-defined shift operator, i.e., $\Sh_a(f)=f\circ \sh_a$.  For a function $f\in\H_0$ we abbreviate $\sh_a(f)=f_a$.  For an operator $T=\begin{pmatrix} f & g\\h& k\end{pmatrix}\in\Mm_2(\H_0)$ we abbreviate 
$T_a=\begin{pmatrix} f _a & g_a\\h_a & k_a\end{pmatrix}$.

We identify $\Linf$ with a subalgebra $\A_0$ of $\B(\H_0)$ in the usual way.  More precisely,
for any $f\in\Linf$ we denote by $\M_f$ the corresponding multiplication operator, i.e., $\M_f(g)=f\cdot g$. Then $\A_0=\{M_f: f\in\Linf\}$. In a similar way, we identify $\Mm_2(\Linf)$ with a subalgebra of $\B(\H)$.   For $T=\begin{pmatrix} f  & g\\h & k\end{pmatrix}\in\Mm_2(\Linf) $ we use the notation 
$$T_{\M}=\begin{pmatrix} \M_f  & \M_g\\ \M_h & \M_k \end{pmatrix} \in\Mm_2(\A_0 )\subseteq \B(\H).$$ 
For any measurable subset $X\subseteq \S$ we will abbreviate $\M_X=\M_{\chi_X}$, where $\chi_X$ denotes the characteristic function of $X$.

We are now ready to define key unitary operators 
$$V_\M,W, U\in\mathcal{M}_2(\B(\H_0))=\B(\H):$$
Let $a$ be a fixed irrational number in the interval $\left(0,\frac{1}{4}\right)$, and let 
$$
J_1=[0,a), J_2=[a,4a),\mbox{ and }J_3=[4a,1).
$$  
We now define
\begin{eqnarray*}
V&=&\frac{1}{\sqrt{2}}\chi_{J_1} \left(\begin{array}{cc}
1 & -1 
\\
 1 & 1 
\end{array}\right) + \frac{1}{\sqrt{2}}\chi_{J_2} \left(\begin{array}{cc}
-1 & -1 
\\
 -\i  & \i 
\end{array}\right)
+\chi_{J_3} \left(\begin{array}{cc}
1 & 0 
\\
 0 & 1 
\end{array}\right),\\
W&=&\Sh_a \left(\begin{array}{cc}
1 & 0 
\\
 0 & 1 
\end{array}\right),
U=V_\M W.
\end{eqnarray*}
We will prove that the self-adjoint abelian algebra 
$$\mathcal{A}=\left\{ \M_f \left(\begin{array}{cc}
1 & 0 
\\
 0 & 1 
\end{array}\right): f\in\Linf\right\} \subseteq \mathcal{M}_2(\H_0)=\B(\H)$$
is $U$-invariant, but is not contained in any $U$-invariant masa $\mathcal{B}$.  We first prove the following.

\begin{lemma}\label{lem-projection}
Any masa $\B$ that contains $\A$ is generated by $\A$ and a projection $P_\M\in\Mm_2(\A_0)$, where for almost every $t\in \S$, $P(t)\in\Mm_2(\mathbb{C})$ is a rank-one projection.  
\end{lemma}
\begin{proof}
We first prove that there is a unitary  $T_\M\in\A'=\Mm(\A_0)$ such that $T_\M \B T_\M^*$ is diagonal. We find such a $T$ explicitly (the existence of such a $T$ may also follow directly from the fact that $\A$ is a von Neumann algebra of uniform multiplicity $2$). Let $B\in \Mm(\Linf)$ be self-adjoint such that $B_\M$ generates $\B$.  With no loss assume that we almost always have  $\tr(B(t))=0$ (replace $B$ by $B-\frac{1}{2}\tr(B)I$ if needed).  On places where $B$ is diagonal, define $T=I$.  On places where $B$ is not diagonal write 
$B=\begin{pmatrix} a & \overline{\xi} b \\ \xi b & -a\end{pmatrix}$, where $a,b\in\mathbb{R}$, $b>0$, and $|\xi|=1$.  Define $x=\frac{a}{\sqrt{a^2+b^2}}$ and let $$T=\frac{1}{\sqrt{2}}\begin{pmatrix} \sqrt{1+x} & \overline{\xi}\sqrt{1-x} \\ 
-\sqrt{1-x} & \overline{\xi}\sqrt{1+x} \end{pmatrix}.$$  Observe that $B=\sqrt{a^2+b^2} B_0$ where $$B_0=\begin{pmatrix} x & \overline{\xi}\sqrt{1-x^2} \\ \xi\sqrt{1-x^2} & -x \end{pmatrix}$$ and that $TB_0T^* = \begin{pmatrix} 1 &0 \\ 0 & -1 \end{pmatrix}$.

We now define $P=T^*\begin{pmatrix} 1 & 0 \\ 0 & 0\end{pmatrix}T$ and note that, by construction, $P(t)$ is always a projection of rank one.
\end{proof}

\begin{theorem}\label{thm-nonexistence} Let $\A$ and $U$ be as defined above.  Then there is no $U$-invariant masa $\B$ that contains $\A$.
\end{theorem}
\begin{proof}
Suppose, toward a contradiction, that there is a $U$-invariant masa $\B$ containing $\A$.  Let $P_\M\in\B$ be a projection such that $P(t)$ is a rank-one projection for almost every $t$ (guaranteed by the Lemma
\ref{lem-projection} above).  Since $U P_\M U^*=(V (P_a) V^*)_\M$ commutes with $P_\M$, we must have that for almost every $t$, 
$$V^*(t)P(a+t)V(t)\in\{P(t),I-P(t)\}.$$  For $S=2P-I$ this is equivalent to the equation
$$
V(t) S(a+t) V^*(t)= \pm S(t),
$$
or, equivalently,
$$
S(a+t)=\pm V^*(t)S(t)V(t).
$$
Note that when $V(t)$ is not a weighted permutation, then neither $S(t)$ nor $S(a+t)$ can be diagonal.  When $V(t)$ is a weighted permutation (e.g., identity), then $S(t)$ is diagonal if and only if $S(a+t)$ is.  Hence the set $X_D\subseteq \S$ of points where $S(t)$ is diagonal is essentially invariant under $\sh_a$.  Since $\sh_a$ is ergodic we therefore must have that $X_D$ has either measure $0$ or measure $1$.  The latter is impossible, since $V(t)$ is not a weighted permutation on $[0,4a)$.  Hence we must have that $S$ is almost never diagonal.  Now write $S=\begin{pmatrix} d & \overline{\theta} e\\
\theta e & -d\end{pmatrix}$ where almost everywhere $d,e\in\mathbb{R}$, $e>0$, and $|\theta|=1$.  We additionally write $\theta=c+\i s$, where $c,s\in\mathbb{R}$ and $c^2+s^2=1$.  

We now examine the equation $$S_a=\pm V^* S V$$ 
on the sets $J_1, J_2, J_3$ (defined while describing $V$) in terms of $d,e,c,s$ by comparing real and imaginary parts of the $(1,1)$ and $(2,1)$ entries of these matrices:

On $J_1$ we get $d_a=\pm \re(\theta) e$ and $\theta_a e_a=\pm(-d+\im(\theta) i e)$, or, equivalently
\begin{eqnarray*}
d_a &=& \pm c e, \\
c_a e_a &=& \mp d, \\
s_a e_a &=& \pm s e.
\end{eqnarray*}

On $J_2$ we get $d_a=\pm \im(\theta) e$ and $\theta_a e_a=\pm(d+\re(\theta) i e)$, or, equivalently
\begin{eqnarray*}
d_a &=& \pm s e, \\
c_a e_a &=& \pm d, \\
s_a e_a &=& \pm c e.
\end{eqnarray*}

On $J_3$ we clearly have
\begin{eqnarray*}
d_a &=& \pm d, \\
c_a e_a &=& \pm c e, \\
s_a e_a &=& \pm s e.
\end{eqnarray*}

Let $$X=\{-1,0,1\}\times \{-1,0,1\}\times \{-1,0,1\}.$$  We define an equivalence relation $\sim$ on $X$ by 
$$(p',q',r')\sim (p,q,r) \iff (p',q',r')=\pm (p,q,r).$$  The set $X$ modulo this equivalence relation is denoted by 
$$X^{\pm}=X/\sim.$$  

Define maps $\alpha_j\colon X^\pm\to X^{\pm}$, $j=1,2,3$ by
\begin{eqnarray*}
\alpha_1(p,q,r)&=&(q,-p,r), \\
\alpha_2(p,q,r)&=&(r,p,q),\\
\alpha_3(p,q,r)&=&(p,q,r).
\end{eqnarray*}

We also define a map $\varphi\colon \S\to X$ by
$$\varphi(t)=\left(\sgn(d(t)), \sgn(c(t)), \sgn(s(t))\right).$$  Here $\sgn\colon\mathbb{R}\to\{-1,0,1\}$ denotes the signum map, i.e., 
$$
\sgn(x)=\begin{cases} 1 & x>0 \\
0 &, x=0 \\
-1 &, x<0 \end{cases}.
$$
We abuse the notation by also using $\varphi$ to denote the corresponding map from $S^1$ to $X^{\pm}$.  Note that $\varphi$ is a well-defined measurable map and that for $t\in J_j$, $j=1,2,3$, we have that 
$$
\varphi(\sh_a(t)) = \alpha_j(\varphi(t)).
$$

For $k=0,1,2,3$ let $X_k^\pm$ denote the subset consisting of those elements $(p,q,r)$ where exactly $k$ of them are zero.  Define also $Y_k=\varphi^{-1}(X_k^\pm)$.  Note that these sets are measurable and they partition $\S$ into a union of disjoint sets. Note that sets $X^\pm_k$, $k=0,\ldots, 3$ are invariant under all $\alpha_j$'s and hence the sets $Y_k$ are essentially invariant under $\sh_a$.  Since $\sh_a$ is ergodic, we therefore must have that one of these sets has measure $1$ and the rest have measure $0$.  We will show that it is impossible for any of these sets to have measure $1$, thereby arriving at a contradiction with the initial assumption that $\A$ can be embedded into a $U$-invariant masa $\B$.

\textbf{Case $\m(Y_3)=1$.}  This is clearly impossible as it would imply that $c=s=0$ almost everywhere (recall that $c^2+s^2=1$).

\textbf{Case $\m(Y_2)=1$.}  Since $c$ and $s$ cannot be zero simultaneously,  we must therefore have that $d=0$ almost everywhere and that almost always $cs$ is zero as well.  So one of the sets $Z_1=\{t\in[0,a): \varphi(t)=(0,\pm 1, 0)\}$, $Z_2=\{t\in[0,a): \varphi(t)=(0,0, \pm 1)\}$ must have strictly positive measure.  But then one of the sets $Z_1'=\sh_{a}(Z_1)$, $Z'_2=\sh_{2a}(Z_2)$ must have a strictly positive measure as well. Note that for $t\in[0,a)$, $\varphi(\sh_{a}(t)) = \alpha_1(\varphi(t))$ and $\varphi(\sh_{2a}(t)) = \alpha_2(\alpha_1(\varphi(t))$.  Hence on both $Z'_1$ and $Z'_2$ we must then have that $c=s=0$ almost everywhere.  But this is clearly impossible, since at least one of these sets must have strictly positive measure.

\textbf{Case $\m(Y_1)=1$.}  Let $b=1-4a$.  Define the map $\sh'_b\colon [0,a)\to [0,a)$ by $sh'_b(t)=b+t\, (\mbox{mod } a\mathbb{Z})$, i.e., $sh'_b(t)$ is the unique $t'\in[0,a)$ such that $\frac{(b+t)-t'}{a}\in\mathbb{Z}$.  Note that $\sh'_b(t)=(\sh_a)^n(t)$ where $n=n(t)\in\mathbb{N}$ is the smallest number such that $(sh_a)^n(t)\in [0,a)$ (here $(\sh_a)^n$ denotes the $n$-fold composition $\sh_a\circ\sh_a\circ\ldots\circ\sh_a$).  Hence for $t\in [0,a)$ we have that $\varphi(\sh'_b(t))=\alpha_1(\varphi(t))$.  This is because $\alpha_2^3=\id$ and $\alpha_3=\id$.  Define $D_1=\{(p,q,0)\in X_1^\pm: pq=1\}, D_2=\{(p,q,0)\in X_1^\pm: pq=-1\}$, $D_3=\{(p,0,r)\in X_1^\pm\}$, $D_4=\{(0,q,r)\in X_1^\pm\}$ and let $E_j=\varphi^{-1}(D_j)\cap [0,a)$ for $j=1,\ldots, 4$.  Since $\alpha_1^4=\id$ we have that these sets are essentially invariant under $(\sh'_b)^4=\sh'_{4b}$.  As this map is ergodic (since $\frac{4b}{a}=\frac{4}{a}-16$ is irrational) we conclude that one of these sets $E_1,\ldots, E_4$ must have measure $a$.  But $E_2=\sh'_b(E_1)$, $E_1=\sh'_b(E_2)$,  $E_4=\sh'_b(E_3)$ and $E_3=\sh'_b(E_4)$, so $\m(E_1)=\m(E_2)$ and $\m(E_3)=\m(E_4)$.  This is clearly impossible.

\textbf{Case $\m(Y_0)=1$.}  The reasoning here is very similar to the argument in the case above.  We again use the map $\sh'_b$ introduced in the case above.  We define sets $F_1=\{(p,q,r)\in X^\pm_0: (pr,qr)=(1,1)\}, F_2=\{(p,q,r)\in X^\pm_0: (pr,qr)=(1,-1)\}$, $F_3=\{(p,q,r)\in X^\pm_0: (pr,qr)=(-1,-1)\}$,
$F_4=\{(p,q,r)\in X^\pm_0: (pr,qr)=(-1,1)\}$ and sets $G_j=\varphi^{-1}(F_j)\cap [0,a)$, $j=1,\ldots, 4$.  Since $\alpha_1(F_1)=F_2$, $\alpha_1(F_2)=F_3$, $\alpha_1(F_3)=F_4$, $\alpha_1(F_4)=F_1$ we conclude that $\sh'_b(G_1)\subseteq G_2, \sh'_b(G_2)\subseteq G_3, \sh'_b(G_3)\subseteq G_4, \sh'_b(G_4)\subseteq G_1$ and hence all sets $G_1,\ldots, G_4$ have equal measure.  But all of these sets are essentially invariant under ergodic $(\sh'_d)^4=\sh'_{4d}$ and hence one of them must have measure $a$; contradicting the above conclusion that they all have equal measure.
\end{proof}

\section*{Acknowledgments}
M.~Mastnak was supported in part by the NSERC Discovery Grant 371994-2019.

\end{document}